\documentclass[11pt]{amsart}
\usepackage{amsmath,amsfonts,amsthm,amssymb,amscd, verbatim, graphicx}

\newcommand{\comments}[1]{}

\renewcommand{\leq}{\leqslant}
\renewcommand{\geq}{\geqslant}

\newcommand{\graphg}{\mathcal{G}}

\newtheorem{prop}{Proposition}
\newtheorem{thm}[prop]{Theorem}
\newtheorem{conj}{Conjecture}

\newtheorem{lem}[prop]{Lemma}

\theoremstyle{definition}

\begin{document}

\title{A note on the {W}eiss conjecture}

\author{Nick Gill}
\address{Department of Mathematics, The Open University, Walton Hall, Milton Keynes, MK7 6AA, UK}
\email{n.gill@open.ac.uk}

\begin{abstract}
Let $G$ be a finite group acting vertex-transitively on a graph. We show that bounding the order of a vertex stabilizer is equivalent to bounding the second singular value of a particular bipartite graph. This yields an alternative formulation of the Weiss Conjecture.
\end{abstract}

\maketitle

Throughout this note $G$ is a finite group acting vertex-transitively on a graph $\Gamma=(V,E)$ of valency $k$. We say that $G$ is {\it locally-}P, for some property P, if $G_v$ is P on $\Gamma(v)$. Here $v$ is a vertex of $\Gamma$, and $\Gamma(v)$ is the set of neighbours of $v$. With this notation we can state the Weiss Conjecture \cite{weiss}.

\begin{conj}\label{c: weiss}
{\rm (The Weiss Conjecture)} There exists a function $f: \mathbb{N}\to \mathbb{N}$ such that if $G$ is vertex-transitive and locally-primitive on a graph $\Gamma$ of valency $k$, then $|G_v|<f(k)$.
\end{conj}

A stronger version of this conjecture, in which `primitive' is replaced by `semiprimitive' has been recently proposed \cite{psv}. (A transitive permutation group is said to be {\it semiprimitive} if each of its normal subgroups is either transitive or semiregular.)

Our aim in this note is to connect the order of $G_v$ to the singular value decomposition of the adjacency matrix of a particular bipartite graph $\graphg$. This connection yields an alternative form of the Weiss conjecture (and its variants). Our main result is the following (we write $\lambda_2$ for the second largest singular value of the adjacency matrix of $\graphg$):

\begin{thm}\label{t: weiss equivalence}
For every function $f:\mathbb{N}\to \mathbb{N}$, there is a function $g: \mathbb{N}\to\mathbb{N}$ such that if $G$ is a finite group acting vertex-transitively on a graph $\Gamma=(V,E)$ of valency $k$ and $\lambda_2<f(k)$, then $|G_v|<g(k)$.

Conversely, for every function $g:\mathbb{N}\to \mathbb{N}$, there is a function $f: \mathbb{N}\to\mathbb{N}$ such that if $G$ is a finite group acting vertex-transitively on a graph $\Gamma=(V,E)$ of valency $k$ and $|G_v|<g(k)$, then $\lambda_2<f(k)$.
\end{thm}

All of the necessary definitions pertaining to Theorem~\ref{t: weiss equivalence} are discussed below. In particular the bipartite graph $\graphg$ is defined in \S1, and the singular value decomposition of its adjacency matrix is discussed in \S2.

Theorem~\ref{t: weiss equivalence} implies that, to any family of vertex-transitive graphs with bounded vertex stabilizer, we have an associated family of bipartite graphs with bounded second singular value, and vice versa. Proving the Weiss Conjecture (or one of its variants) is, therefore, equivalent to bounding the second singular value for a particular family of bipartite graphs.

Gowers remarks that singular values are the `correct analogue of eigenvalues for bipartite graphs' (see the preamble to Lemma 2.7 in \cite{gowers}).\footnote{The mathematics behind this remark is set down in \cite{bn}. An elementary first observation is that the eigenvalues of the natural adjacency matrix of a bipartite graph may be negative, in contrast to the eigenvalues of the (symmetric) adjacency matrix of a graph. This pathology is remedied by studying the singular values as we shall see.} Thus bounding the second singular value of a bipartite graph is analogous to bounding the second eigenvalue of a graph; the latter task is a celebrated and much studied problem due to its connection to the expansion properties of a graph (see, for instance, \cite{lubotz}).

The fact that the Weiss Conjecture has connections to expansion has already been recognised \cite{ppss} - we hope that this note adds to the evidence that it is a connection warranting a good deal more investigation. 

\subsection{The associated bipartite graph $\graphg$}\label{s: bipartite}

Our first job is to describe $\graphg$, and for this we need the concept of a {\it coset graph}. Let $H$ be a subgroup of $G$ and let $A$ be a union of double cosets of $H$ in $G$ such that $A=A^{-1}$. Define the coset graph ${\mathrm Cos}(G,H,A)$ as the graph with vertex set the left cosets of $H$ in $G$ and with edges the pairs $\{xH,yH\}$ such that $Hx^{-1}yH\subset A$. Observe that the action of $G$ by left multiplication on the set of left cosets of $H$ induces a vertex-transitive automorphism group of ${\mathrm Cos}(G,H,A)$.

The following result is due to Sabidussi \cite{sabidussi}.

\begin{prop}
Let $\Gamma=(V,E)$ be a $G$-vertex-transitive graph and $v$ a vertex of $\Gamma$. Then there exists a union $S$ of $G_v$-double cosets such that $S=S^{-1}$,  $\Gamma \cong {\mathrm Cos}(G,G_v,S)$ and the action of $G$ on $V$is  equivalent to the action of $G$ by left multiplication on the left cosets of $G_v$ in $G$.
\end{prop}

Note that $G$ is locally-transitive if and only if $S$ is equal to a single double coset of $G_v$. From here on we fix $v$ to be a vertex in $V$ and we set $S$ to be the union of double cosets of $G_v$ in $G$ such that $\Gamma \cong {\mathrm Cos}(G,G_v,S)$. Observe that $S(\{v\}) = \Gamma(v)$.

We are ready to define the regular bipartite graph $\graphg$. We define the two vertex sets, $X$ and $Y$, to be copies of $V$. The number of edges between $x\in X$ and $y\in Y$ is defined to equal the number of elements $s\in S$ such that $s(x)=y$. Note that $\graphg$ is a multigraph and may contain loops.

\subsection{The singular value decomposition}\label{s: svd}

For $V$ and $W$ two real inner product spaces, we define a linear map
$$w\otimes v: V\to W, x \mapsto \langle x,v\rangle w.$$
With this notation we have the following result \cite[Theorem 2.6]{gowers}.

\begin{prop}\label{p: svd}
Let $\alpha:V\to W$ be a linear map. Then $\alpha$ has a decomposition of the form $\sum_{i=1}^k \lambda_i w_i\otimes v_i$, where the sequences $(v_i)$ and $(w_i)$ are orthonormal in $V$ and $W$, respectively, each $\lambda_i$ is non-negative, and $k$ is the smaller of $\dim V$ and $\dim W$.
\end{prop}

The decomposition described in the proposition is called the {\it singular value decomposition}, and the values $\lambda_1, \lambda_2,\dots$ are the {\it singular values} of $\alpha$. In what follows we always assume that the singular values are written in non-increasing order: $\lambda_1\geq \lambda_2 \geq \cdots $.

Now write $\mathcal{A}$ for the adjacency matrix of $\graphg$ \emph{as a bipartite graph}, i.e. the rows of $\mathcal{A}$ are indexed by $X$, the columns by $Y$ and, for $x\in X, y\in Y$, the entry $\mathcal{A}(x,y)$ is equal to the number of edges between $x$ and $y$. Then $\mathcal{A}$ can be thought of as a matrix for a linear map $\alpha: \mathbb{R}^{X} \to \mathbb{R}^{Y}$ and, as such, we may consider its singular value decomposition. From here on the variables $\lambda_1, \lambda_2, \dots$ will denote the singular values of this particular map.

The next result gives information about this decomposition. (The result is \cite[Lemma 3.3]{gillqr}, although some of the statements must be extracted from the proof.)

\begin{lem}\label{l: svd alpha}
\begin{enumerate}
 \item $\lambda_1 = t\sqrt{|V_1||V_2|}$ where $t$ is the real number such that every vertex in $V_1$ has degree $p|V_2|$.
 \item If $f$ is a function that sums to zero, then $\|\alpha(f)\|/ \|f\| \leq \lambda_2$.
\end{enumerate}
 \end{lem}

Note that the only norm used in this note is the $\ell^2$-norm.

\subsection{Convolution}

Consider two functions $\mu:G\to\mathbb{R}$ and $\nu:V\to\mathbb{R}$. We define the \emph{convolution} of $\mu$ and $\nu$ to be
\begin{equation}\label{e: convolution definition}
\mu\ast \nu: V\to \mathbb{R}, \, \, v\mapsto \sum\limits_{g\in G} \mu(g)\nu(g^{-1}v).
\end{equation}

In the special case where $\mu = \chi_S$, the characteristic function of the set $S$ defined above, $\chi_S \ast \nu$ takes on a particularly interesting form:
\begin{equation}\label{e: alpha 2}
(\chi_S\ast f)(v)  = \sum\limits_{g\in G} \chi_S(g)f(g^{-1}v) = \sum\limits_{w\in \Omega} \mathcal{A}(v,w)f(v).
\end{equation}

Here, as before, $\mathcal{A}$ is the adjacency matrix of the bipartite graph $\mathcal{G}$. Equation~\eqref{e: alpha 2} implies that the linear map $\alpha: \mathbb{R}^{X} \to \mathbb{R}^{Y}$, for which $\mathcal{A}$ is a matrix, is given by $\alpha(f)=\chi_S\ast f$. This form is particularly convenient, as it allows us to use the following easy identities \cite[Lemma 2.3]{gillqr}.

\begin{lem}\label{l: 32}
Let $f$ be a function on $V$ that sums to $0$, $p$ a probability distribution over $V$, $q$ a probability distribution over $G$, and $U$ the uniform probability distribution over $V$. Then
\begin{enumerate}
\item $\Vert f+U\Vert^2 = \Vert f \Vert^2 + \frac{1}{|\Omega|}.$
 \item $\Vert p-U\Vert^2 = \Vert p \Vert^2 - \frac{1}{|\Omega|}.$
\item $\Vert q\ast (p\pm U)\Vert  = \Vert q\ast p\pm U\Vert.$
\item For $k$ a real number, $\Vert k p \Vert = k\Vert p\Vert.$
\end{enumerate}
\end{lem}

\subsection{The proof}

Theorem~\ref{t: weiss equivalence} will follow from the next result which shows that, provided $k$ is not too large compared to $|V|$, the order of $G_v$ is bounded in terms of $\lambda_2$ and $k$.

\begin{prop}\label{p: lambda2 implies weiss}
Either $|G_v|<\frac{\sqrt{2}\lambda_2}{k}$ or $|V|< 2k$.
\end{prop}
\begin{proof}
Let $v$ be a vertex in $V$. We define two probability distributions, $p_S: G\to \mathbb{R}$ and $p_v: V\to \mathbb{R}$, as follows:
\[
p_S(x) = 
\begin{cases}
\frac{1}{|S|}, & x\in S, \\
0, & x\notin S,
\end{cases}
\qquad          
p_v(x) = 
\begin{cases}
1, & x=v, \\
0, & otherwise.
\end{cases}
\]          
Observe that $\Vert p_S\Vert = \frac1{\sqrt{|S|}}=\frac 1{\sqrt{k|G_v|}}$ and $\Vert p_\Gamma\Vert = 1$. Observe that $(p_S\ast p_v)(w)=0$ except when $w\in S(\{v\})=\Gamma(v)$. A simple application of the Cauchy-Schwarz inequality (or see \cite[Observation 3.4]{bnp}) gives
\[
\frac1k=\frac{1}{|\Gamma(v)|} \leq \Vert p_S\ast p_\Gamma \Vert^2. 
\]

Define $f=p_v - U$ and observe that $f$ is a function on $V$ that sums to $0$. Lemma~\ref{l: svd alpha} implies that $\|(\alpha f)\|/\|f\| \leq \lambda_2$. Using this fact, the identities in Lemma~\ref{l: 32}, and the fact that $\chi_S = |S| p_S$, we obtain the following:
\begin{equation*}
 \begin{aligned}
  \frac{1}{k} &\leq \Vert p_S\ast p_v \Vert^2 \\
&=\| p_S\ast (f+U)\|^2 \\
&=\|p_S \ast f + U\|^2 \\
&=\|p_S \ast f\|^2 + \frac1{|V|}\\
&=\frac1{|S|^2}\|\chi_S\ast f\|^2 + \frac1{|V|}\\
&=\frac1{|S|^2}\|\alpha(f)\|^2 + \frac1{|V|}\\
&\leq \frac1{|S|^2}\lambda_2^2 \Vert f\Vert^2 + \frac{1}{|V|}\\
&= \frac1{|S|^2}\lambda_2^2 \Vert p_v-U \Vert^2 +\frac{1}{|V|}\\
&< \frac{\lambda_2^2}{|S|^2}+\frac{1}{|V|}.
 \end{aligned}
\end{equation*}
Since $|S|=k|G_v|$ we can rearrange to obtain  
\[
k> \frac{|V|}{1+\frac{\lambda_2^2|V|}{k^2|G_v|^2}}.
\]
Observe that if $\frac{\lambda_2^2|V|}{k^2|G_v|^2}\leq 1$ then 
\[
k>\frac{|V|}{1+\frac{\lambda_2^2|V|}{k^2|G_v|^2}} \geq \frac{|V|}{2}.
\]
and the result follows. On the other hand, if $\frac{\lambda_2^2|V|}{k^2|G_v|^2}> 1$ then 
\[
k> \frac{|V|k^2|G_v|^2}{k^2|G_v|^2 + |V|\lambda_2^2} > \frac{|V|k^2|G_v|^2}{2|V|\lambda_2^2}
\]
and we conclude that $|G_v|^2 < 2\lambda_2^2/k$ as required.
\end{proof}

Finally we can prove Theorem~\ref{t: weiss equivalence}.

\begin{proof}
 The previous lemma implies that if $\lambda_2<f(k)$ for some function $f:\mathbb{N}\to \mathbb{N}$ then $|G_v|<g(k)$ for some function $g:\mathbb{N}\to\mathbb{N}$. (Note that if $|V|\leq 2k$, then $|G_v|\leq |G| \leq (2k)!$.)

For the converse, Lemma~\ref{l: svd alpha} implies that $\lambda_1 = t\sqrt{|X|\cdot |Y|}$ where $t$ is the real number such that every vertex in $X$ has degree $t|Y|$. Now recall that $|X|=|Y|=|V|$ and observe that every vertex in $X$ has degree $k|G_v|$. Thus we conclude that $\lambda_1 = k|G_v|$. Since $\lambda_2\leq \lambda_1$ the result follows.
\end{proof}

\providecommand{\bysame}{\leavevmode\hbox to3em{\hrulefill}\thinspace}
\providecommand{\MR}{\relax\ifhmode\unskip\space\fi MR }
\providecommand{\MRhref}[2]{%
  \href{http://www.ams.org/mathscinet-getitem?mr=#1}{#2}
}
\providecommand{\href}[2]{#2}


\begin{thebibliography}{PPSS12}

\bibitem[BN04]{bn}
B.~Bollob{\'a}s and V.~Nikiforov, \emph{Hermitian matrices and graphs: singular
  values and discrepancy}, Discrete Math. \textbf{285} (2004), no.~1-3, 17--32.

\bibitem[BNP08]{bnp}
L.~Babai, N.~Nikolov, and L.~Pyber, \emph{Product growth and mixing in finite
  groups}, Proceedings of the {N}ineteenth {A}nnual {ACM}-{SIAM} {S}ymposium on
  {D}iscrete {A}lgorithms (New York), ACM, 2008, pp.~248--257.

\bibitem[Gil]{gillqr}
N.~Gill, \emph{Quasirandom group actions}, 2013. Submitted. Preprint available
  on the Math arXiv: {\tt http://arxiv.org/abs/1302.1186}.

\bibitem[Gow08]{gowers}
W.~T. Gowers, \emph{Quasirandom groups}, Comb. Probab. Comp. \textbf{17}
  (2008), 363--387.

\bibitem[Lub10]{lubotz}
A.~Lubotzky, \emph{Discrete groups, expanding graphs and invariant measures},
  Modern Birkh\"auser Classics, Birkh\"auser Verlag, Basel, 2010, With an
  appendix by Jonathan D. Rogawski, Reprint of the 1994 edition.

\bibitem[PPSS12]{ppss}
C.~Praeger, L.~Pyber, P.~Spiga, and E.~Szab\'o, \emph{Graphs with automorphism
  groups admitting composition factors of bounded rank}, Proc. Amer. Math. Soc.
  \textbf{140} (2012), no.~7, 2307--2318.

\bibitem[PSV12]{psv}
P.~Poto{\v{c}}nik, P.~Spiga, and G.~Verret, \emph{On graph-restrictive
  permutation groups}, J. Combin. Theory Ser. B \textbf{102} (2012), no.~3,
  820--831.

\bibitem[Sab64]{sabidussi}
G.~Sabidussi, \emph{Vertex-transitive graphs}, Monatsh. Math. \textbf{68}
  (1964), 426--438.

\bibitem[Wei81]{weiss}
R.~Weiss, \emph{{$s$}-transitive graphs}, Algebraic methods in graph theory,
  {V}ol. {I}, {II} ({S}zeged, 1978), Colloq. Math. Soc. J\'anos Bolyai,
  vol.~25, North-Holland, Amsterdam, 1981, pp.~827--847.

\end{thebibliography}
\end{document}